\documentclass[reqno]{amsart}
\usepackage{latexsym,amssymb}
\usepackage{upref}
\usepackage{graphicx}
\usepackage[active]{srcltx}
\usepackage[colorlinks=printing,backref]{hyperref}      % Print version

 \makeatletter
\renewcommand*\subjclass[2][2000]{%
  \def\@subjclass{#2}%
  \@ifundefined{subjclassname@#1}{%
    \ClassWarning{\@classname}{Unknown edition (#1) of Mathematics
      Subject Classification; using '1991'.}%
  }{%
    \@xp\let\@xp\subjclassname\csname subjclassname@#1\endcsname
  }%
}
 \makeatother
\usepackage{enumerate,url,amssymb,  mathrsfs}%, pdfsync}% ,showkeys, pdfsync}
\newtheorem{theorem}{Theorem}[section]
\newtheorem{lemma}[theorem]{Lemma}
\newtheorem{corollary}[theorem]{Corollary}
\newtheorem{proposition}[theorem]{Proposition}
\theoremstyle{definition}
\newtheorem{definition}[theorem]{Definition}

\numberwithin{equation}{section}
\newcommand{\abs}[1]{\lvert#1\rvert}

\begin{document}

\title{Harmonic mappings and distance function}

\subjclass{58E20,30C62}

\keywords{Harmonic mappings, Quasiconformal mappings, Distance
function, Hopf lemma}

\author{David Kalaj}
\address{University of Montenegro, Faculty of Natural Sciences and
Mathematics, Cetinjski put b.b. 81000 Podgorica, Montenegro}
\email{davidk@t-com.me}
\begin{abstract}
We prove the following theorem: every quasiconformal harmonic
mapping between two plane domains with $C^{1,\alpha}$
 ($\alpha<1$), respectively $C^{1,1}$ compact
boundary
 is bi-Lipschitz. The distance function with respect to the
boundary of the image domain is used. This in turn extends a similar
result of the author in \cite{kalajan} for Jordan domains, where
stronger boundary conditions for the image domain were needed.
\end{abstract} \maketitle

%\tableofcontents

\section{Introduction and statement  of the main result}
We say that a real function $u:D\to \mathbf R$ is ACL (absolutely
continuous on lines) in  the region $D$, if for every closed
rectangle $R\subset D$ with sides parallel to the $x$ and $y$-axes,
$u$ is absolutely continuous on a.e. horizontal and a.e. vertical
line in $R$. Such a function has of course, partial derivatives
$u_x$, $u_y$ a.e. in $D$. A  homeomorphism  $f\colon D\mapsto G, $
where $D$ and $G$ are subdomains of the complex plane $\mathbf C,$
is said to be $K$-quasiconformal (K-q.c), $K\ge 1$, if $f$ is ACL
and
\begin{equation}\label{defqc} |\nabla f(z)|\le K l(\nabla f(z))\quad \text{a.e. on $D$},
\end{equation} where $$|\nabla f(x)|:=\max_{|h|=1}|\nabla f(x)h|= |f_z|+|f_{\bar z}|$$
and $$l(\nabla f(z)):=\min_{|h|=1}|\nabla f(z)h| =|f_z|-|f_{\bar
z}|$$
  (cf. \cite[p.23--24]{Ahl} and  \cite{lv}). Note that, the
condition (\ref{defqc}) can be written as $$|f_{\bar z}|\le
k|f_z|\quad \text{a.e. on $D$ where $k=\frac{K-1}{K+1}$ i.e.
$K=\frac{1+k}{1-k}$ }$$ or in its equivalent form
\begin{equation}\label{nj}|\nabla f(z)|^2\le K J_f(z), \, z\in \Bbb U,\end{equation} where $J_f$ is the jacobian of $f$.

A function $w$ is called \emph{harmonic} in a region $D$ if it has
form $w=u+iv$ where $u$ and $v$ are real-valued harmonic functions
in $D$. If $D$ is simply-connected, then there are two analytic
functions $g$ and $h$ defined on $D$ such that $w$ has the
representation
$$w=g+\overline h.$$

If $w$ is a harmonic univalent function, then by Lewy's theorem (see
\cite{hl}), $w$ has a non-vanishing Jacobian and consequently,
according to the inverse mapping theorem, $w$ is a diffeomorphism.

Let $$P(r,x)=\frac{1-r^2}{2\pi (1-2r\cos x+r^2)}$$ denote the
Poisson kernel. Then every bounded harmonic function $w$ defined on
the unit disc $\mathbf U:=\{z:|z|<1\}$ has the following
representation
\begin{equation}\label{e:POISSON}
w(z)=P[F](z)=\int_0^{2\pi}P(r,x-\varphi)F(e^{ix})dx,
\end{equation}
where $z=re^{i\varphi}$ and $F$ is a bounded integrable function
defined on the unit circle $S^1$.

In this paper we continue to study q.c. harmonic mappings.
 See \cite{Om} for the pioneering work on this topic
and see \cite{hs} for related earlier results. In some recent
papers, a lot of work have been done on this class of mappings
 (\cite{akm}, \cite{kalajan}- \cite{MMM}, \cite{MP}, \cite{pk}, \cite{kojic} and \cite{mv}). In
these papers it is established the Lipschitz and the co-Lipschitz
character of q.c. harmonic mappings between plane domains with
certain boundary conditions. In \cite{wan} it is considered the same
problem for hyperbolic harmonic quasiconformal selfmappings of the
unit disk. Notice that, in general, quasi-symmetric self-mappings of
the unit circle do not provide quasiconformal harmonic extension to
the unit disk. In \cite{Om} it is given an example of $C^1$
diffeomorphism of the unit circle onto itself, whose Euclidean
harmonic extension is not Lipschitz. Alessandrini and Nessi proved
in \cite{ale} the following proposition:
\begin{proposition}\label{ales} Let $F: S^1\to \gamma\subset \Bbb C$
be an orientation preserving diffeomorphism of class $C^1$ onto a
simple closed curve. Let $D$ be the bounded domain such that
$\partial D =\gamma$. Let $w=P[F]\in C^1(\overline{\Bbb U};\Bbb C).$
The mapping $w$ is a diffeomorphism of $\Bbb U$ onto $D$ if and only
if \begin{equation}\label{use} J_w
> 0\text{ everywhere on}\  S^1.\end{equation}
\end{proposition}
In view of the inequalities \eqref{nj} and \eqref{use}, we easily
see that.
\begin{corollary}
Under the condition of Proposition~\ref{ales}, the harmonic mapping
$w$ is a diffeomorphism if and only if it is $K$ quasiconformal for
some $K\ge 1$.
\end{corollary}

In contrast to the Euclidean metric, in the case of hyperbolic
metric, if $f:S^1\mapsto S^1$ is $C^1$ diffeomorphism, or more
general if $f:S^{n-1}\mapsto S^{m-1}$ is a mapping with the
non-vanishing energy, then its hyperbolic harmonic extension is
$C^1$ up to the boundary (\cite{lit}) and (\cite{lit1}).

To continue we need the definition of $C^{k,\alpha}$ Jordan curves
($k\in \Bbb N$, $0<\alpha\le 1$). Let $\gamma$ be a rectifiable
curve in the complex plane. Let $l$ be the length of $\gamma$. Let
$g:[0,l]\mapsto \gamma$ be an arc-length parametrization of
$\gamma$. Then $|\dot g(s)|=1 \text{ for all $s\in [0,l]$}.$ {We
will write that the curve $\gamma\in C^{k,\alpha}$, $k\in \Bbb N$,
$0<\alpha\le 1$ if $g\in C^k$, and $M(k,\alpha):=\sup_{t\neq
s}\frac{|g^{(k)}(t)-g^{(k)}(s)|}{|t-s|^{\alpha}}<\infty.$ Notice
this important fact, if $\gamma\in C^{1,1}$ then $\gamma$ has the
curvature $\kappa_z$ for a.e. $z\in\gamma$ and
$\mathrm{ess\,sup}\{|\kappa_z|: z\in \gamma\}\le M(1,1)<\infty$.

This definition can be easily extended to arbitrary $C^{k,\alpha}$
compact $1-$ dimensional manifold (not necessarily connected).

The starting point of this paper is the following proposition.

\begin{proposition}\label{kr}
Let $w=f(z)$ be a $K$ quasiconformal harmonic mapping between a
 Jordan domain $\Omega_1$ with $C^{1,\alpha}$ boundary and a
Jordan domain $\Omega$ with $C^{1,\alpha}$  (respectively $C^{2,\alpha}$) boundary. Let in addition
$b\in \Omega_1$ and $a=f(b)$. Then $w$ is Lipschitz (co-Lipschitz). Moreover
there exists a positive constant $c=c(K,\Omega,\Omega_1, a, b)\ge 1$
such that
\begin{equation}\label{lipschitz}
|f(z_1)-f(z_2)|\le c|z_1-z_2|,\,\,\,\,z_1,z_2\in
\Omega_1\end{equation}

and

\begin{equation}\label{colipschitz} \frac 1c |z_1-z_2|\le
|f(z_1)-f(z_2)|,\,\,\,\,z_1,z_2\in
\Omega_1,\end{equation}
respectively.
\end{proposition}
See \cite{kamz} for the first part of Proposition~\ref{kr} and \cite{kalajan} for its second part.

In \cite{kalajan}, it was conjectured that the second part of Proposition~\ref{kr} remains hold
if we assume that $\Omega$ has $C^{1,\alpha}$ boundary only.

Notice that the proof of Proposition~\ref{kr} relied on  Kellogg-Warschawski theorem
(\cite{w1}, \cite{w2}, \cite{G}) from the theory of conformal
mappings, which asserts that if $w$ is a conformal mapping of the
unit disk onto a domain $\Omega\in C^{k,\alpha}$, then $w^{(k)}$ has
a continuous extension to the boundary ($k\in \Bbb N$). It also depended on the Mori's theorem
from the theory of q.c. mappings, which diels with H\"older character of q.c. mappings between
plane domains (see\cite{Ahl} and \cite{wang}). In addition,
Lemma~\ref{hopf} were needed.

Using a different approach, we extend the second part of
Proposition~\ref{kr} to the class of image domains with $C^{1,1}$
boundary. Its extension is Theorem~\ref{krye}. The proof of
Theorem~\ref{krye}, given in the last section, is different form the
proof of second part of Proposition~\ref{kr}, and the use of
Kellogg-Warschawski theorem for the second derivative (\cite{w2}) is
avoided. The distance function is used and thereby a ``weaker"
smoothness of the boundary of image domain is needed.

\begin{theorem}[The main theorem]\label{krye}
Let $w=f(z)$ be a $K$ quasiconformal harmonic mapping of the unit disk $\Bbb U$
and a Jordan domain $\Omega$ with $C^{1,1}$ boundary. Let in addition
 $a=f(0)$. Then $w$ is co-Lipschitz. More precisely
there exists a positive constant $c=c(K,\Omega, a)\ge 1$
such that
\begin{equation}\label{pocetna} \frac 1c |z_1-z_2|\le
|f(z_1)-f(z_2)|,\,\,\,\,z_1,z_2\in
\Omega.\end{equation}
\end{theorem}
Since the composition of a q.c. harmonic and a conformal mapping is
itself q.c. harmonic, using Theorem~\ref{krye} and Kellogg's theorem
for the first derivative we obtain:
\begin{corollary}\label{coco}Let $w=f(z)$ be a $K$ quasiconformal harmonic mapping between a
 plane domain $\Omega_1$ with $C^{1,\alpha}$ compact boundary and a
plane domain $\Omega$ with $C^{1,1}$ compact boundary. Let in
addition $a_0\in \Omega_1$ and $b_0=f(a_0)$. Then $w$ is
bi-Lipschitz. Moreover there exists a positive constant
$c=c(K,\Omega,\Omega_1, a_0, b_0)\ge 1$ such that
\begin{equation}\label{finish} \frac 1c |z_1-z_2|\le
|f(z_1)-f(z_2)|\le c|z_1-z_2|,\,\,\,\,z_1,z_2\in
\Omega_1.\end{equation}
\end{corollary}
\begin{proof}[Proof of corollary~\ref{coco}]
Let $b=f(a)\in \partial\Omega$. As $\partial \Omega\in C^{1,1}$, it
follows that there exists a $C^{1,1}$ Jordan curve $\gamma_b\subset
\overline\Omega$, whose interior $D_b$ lies in $\Omega$, and
$\partial\Omega\cap \gamma_b$ is a neighborhood of $b$. See
\cite[Theorem~2.1]{kamz} for an explicit construction of such Jordan
curve. Let $D_a = f^{-1}(D_b)$, and take a conformal mapping $g_a$
of the unit disk onto $D_a$. Then $f_a = f\circ g_a$ is a q.c.
harmonic mapping of the unit disk onto the $C^{1,1}$ domain $D_b$.
According to Theorem~\ref{krye} it follows that $f_a$ is
bi-Lipschitz. According to Kellogg's theorem, it follows that
$f=f_a\circ g_a^{-1}$ and its inverse $f^{-1}$ are Lipschitz in some
small neighborhood of $a$ and of $b = f(a)$ respectively. This means
that $\nabla f$ is bounded in some neighborhood of $a$. Since
$\partial \Omega_1$ is a compact, we obtain that $\nabla f$ is
bounded in $\partial \Omega_1$. The same hold for $\nabla f^{-1}$
with respect to $\partial \Omega$. This in turn implies that $f$ is
bi-Lipschitz.
\end{proof}
\section{Auxiliary results}
Let $\Omega$ be a domain in $\Bbb R^2$ having non-empty boundary
$\partial \Omega$. The distance function is defined by
\begin{equation}\label{dist}
d(x)=\mathrm{dist}\,(x,\partial\Omega).
\end{equation}

Let $\Omega$ be bounded and $\partial\Omega\in C^{1,1}$. The
conditions on $\Omega$ imply that $\partial\Omega$ satisfies the
following condition: at a.e. point $z\in \partial\Omega$ there
exists a disk $D = D(w_{z}, r_z)$ depending on $z$ such that
$\overline D\cap (\Bbb C\setminus \Omega) = \{z\}$. Moreover $\mu :
= \mathrm{ess}\inf\{r_z, z\in \partial \Omega\}>0$. It is easy to
show that $\mu^{-1}$ bounds the curvature of $\partial \Omega$,
which means that $\frac{1}{\mu} \ge {\kappa_z},$ for $z\in
\partial\Omega$. Here $\kappa_z$ denotes the curvature of
$\partial\Omega$ at $z\in\partial\Omega$. Under the above
conditions, we have $d\in C^{1,1}(\Gamma_\mu)$, where
$\Gamma_\mu=\{z\in\overline{\Omega}:d(z)<\mu\}$ and for $z\in
\Gamma_\mu$ there exists $\omega(z)\in \partial\Omega$ such that
\begin{equation}\label{distnorm} \nabla d(z)=\mathbf{\nu}_{\omega(z)},
\end{equation}
where $\mathbf{\nu}_{\omega(z)}$ denotes the inner normal vector to
the boundary $\partial\Omega$
 at the point $\omega(z)$.
See \cite[Section~14.6]{gt} for details.

\begin{lemma}
Let $w:\Omega_1\mapsto \Omega$ be a $K$ q.c. and $\chi=-d(w(z))$.
Then
\begin{equation}\label{dqh} |\nabla\chi|\le|\nabla w|\le K|\nabla
\chi|
\end{equation}
in $w^{-1}(\Gamma_\mu)$ for $\mu>0$ such that
$1/\mu>\kappa_0=\mathrm{ess}\
\sup\{|\kappa_z|:z\in\partial\Omega\}$.
\end{lemma}
\begin{proof}
Observe first that $\nabla d$ is a unit vector. From $\nabla
\chi=-\nabla d\cdot \nabla w$ it follows that
$$|\nabla\chi|\le|\nabla d||\nabla w|=|\nabla w|.$$ For a non-singular matrix $A$ we have
\begin{equation}\begin{split}\inf_{|x|=1}
|Ax|^2&=\inf_{|x|=1}\left<Ax,Ax\right>=\inf_{|x|=1}\left<A^T A
x,x\right>\\&=\inf\{\lambda: \exists x\neq 0, A^T A x=\lambda
x\}\\&=\inf\{\lambda: \exists x\neq 0, A A^T A x=\lambda
Ax\}\\&=\inf\{\lambda: \exists y\neq 0, A A^T y=\lambda
y\}=\inf_{|x|=1} |A^T x|^2.\end{split}\end{equation} Next we have
that $(\nabla \chi)^T=-(\nabla w)^T\cdot (\nabla d)^T$ and therefore
for $x\in w^{-1}(\Gamma_\mu)$, we obtain $$|\nabla \chi|\ge
\inf_{|e|=1}|(\nabla w)^T \,e|=\inf_{|e|=1}|\nabla w \,e|=l(w)\ge
K^{-1}|\nabla w|.$$

The proof of (\ref{dqh}) is completed.

\end{proof}
\begin{lemma}\label{pipi} Let $\{e_1, e_2\}$ be the
natural basis in the space $\mathbf R^2$. Let $w:\Omega_1\mapsto
\Omega$ be a twice differentiable mapping and let $\chi=-d(w(z))$.
Then
\begin{equation}\label{lapdistance}\Delta \chi(z_0) =
\frac{\kappa_{w_0}}{1-\kappa_{w_0} d(w(z_0))}|(O_{z_0} \nabla
w(z_0))^Te_1|^2 -\left<(\nabla d)(w(z_0)), \Delta
w\right>,\end{equation} where $z_0 \in w^{-1}(\Gamma_\mu)$,
$\omega_0\in\partial\Omega$ with
$|w(z_0)-\omega_0|=\mathrm{dist}(w(z_0),
\partial\Omega)$,  $\mu>0$
such that $1/\mu>\kappa_0=
\mathrm{ess}\sup\{|\kappa_z|:z\in\partial\Omega\}$ and $O_{z_0}$ is
an orthogonal transformation.
\end{lemma}
\begin{proof}
 Let $\nu_{\omega_0}$ be the inner unit normal vector of
$\gamma$ at the point $\omega_0\in \gamma$. Let  $O_{z_0}$ be an
orthogonal transformation that takes the vector $e_2$ to
$\nu_{\omega_0}$. In complex notations:
$$O_{z_0}w=-i\nu_{\omega_0}w.$$
Take $\tilde\Omega:=O_{z_0}\Omega$. Let $\tilde d$ be the distance
function with respect to $\tilde\Omega$. Then $$d(w) = \tilde
d(O_{z_0}w) = \mathrm{dist}\ (O_{z_0}w, \partial\tilde \Omega).$$
Therefore $\chi(z)=-\tilde d(O_{z_0}(w(z)))$.
\\
Furthermore
\begin{equation}\label{hess}\begin{split}
\Delta\chi(z)&=-\sum_{i=1}^2(D^2\tilde
d)(O_{z_0}(w(z)))(O_{z_0}\nabla w(z)e_i,O_{z_0}\nabla
w(z)e_i)\\&-\left<\nabla d(w(z)),\Delta w(z)\right>.
\end{split}
\end{equation}
To continue, we make use of the following proposition.
\begin{proposition}\cite[Lemma~14.17]{gt}
Let $\Omega$ be bounded and $\partial\Omega\in C^{1,1}$.
 Then under notation of Lemma~\ref{pipi} we have
\begin{equation}\label{diag} (D^2\tilde
d)(O_{z_0}w(z_0))=\mathrm{diag}(\frac{-\kappa_{\omega_0}}{1-\kappa_{\omega_0}
d},0)=\begin{pmatrix}
 \frac{-\kappa_{\omega_0}}{1-\kappa_{\omega_0} d} & 0 \\
  0 & 0
\end{pmatrix},
\end{equation}
where $\kappa_{\omega_0}$ denotes the curvature of $\partial\Omega$
at $\omega_0\in
\partial\Omega$.
\end{proposition}

 Applying (\ref{diag}) we have
\begin{equation}\label{ende}\begin{split}
&\sum_{i=1}^{2}(D^2\tilde d)(O_{z_0}(w(z_0)))(O_{z_0}(\nabla
w(z_0))e_i,O_{z_0}(\nabla w(z_0))e_i)\\&=\sum_{i=1}^2
\sum_{j,k=1}^2D_{j,k}\tilde d(O_{z_0}(w(z_0))) \,D_i
(O_{z_0}w)_j(z_0)\cdot D_i (O_{z_0}w)_k(z_0)\\&=\sum_{j,k=1}^2
D_{j,k}\tilde d (O_{z_0}(w(z_0)))\left<(O_{z_0} \nabla
w(z_0))^Te_j,(O_{z_0} \nabla w(z_0))^Te_k\right>\\&=
\frac{-\kappa_{\omega_0}}{1-\kappa_{\omega_0}\tilde d}|(O_{z_0}
\nabla w(z_0))^Te_1|^2.
\end{split}
\end{equation}
Finally we obtain
\begin{equation*}\label{lapdis}\Delta \chi(z_0) =
\frac{\kappa_{\omega_0}}{1-\kappa_{\omega_0}\tilde d}|(O_{z_0}
\nabla w(z_0))^Te_1|^2 -\left<(\nabla d)(w(z_0)), \Delta
w\right>.\end{equation*}
\end{proof}
\section{The proof of the main theorem}
The main step in proving the main theorem is the following lemma.
\begin{lemma}\label{lemica}
Let $w=f(z)$ be a $K$ quasiconformal mapping of the unit disk onto a
$C^{1,1}$ Jordan domain $\Omega$ satisfying the differential
inequality \begin{equation}\label{pde}|\Delta w|\le B|\nabla w|^2,\,
\, B\ge 0\end{equation} for some $B\ge 0$. Assume in addition that
$w(0)=a_0\in \Omega $. Then there exists a constant
$C(K,\Omega,B,a)>0$ such that
\begin{equation}\label{eq}\abs{\frac{\partial w}{\partial r}(t)}\ge
C(K,\Omega,B,a_0)\text{ for almost every } t\in S^1.\end{equation}
\end{lemma}
\begin{proof}
Let us find $A>0$ so that the function $\varphi_w(z)=
-\frac{1}{A}+\frac{1}{A}e^{-Ad(w(z))}$ is subharmonic on $\{z:
d(w(z))<\frac{1}{2\kappa_0}\}$, where
$$\kappa_0=\mathrm{ess} \sup \{|\kappa_w|:w\in \gamma\}.$$

Let $\chi = -d(w(z))$. Combining \eqref{dqh}, \eqref{lapdistance}
and \eqref{pde} we get
\begin{equation}\label{delc}|\Delta \chi|\le{2\kappa_0}|\nabla
w|^2 + B|\nabla w|^2  \le({2\kappa_0}+B)K^2|\nabla \chi|^2
.\end{equation} Take
$$g(t)=-\frac{1}{A}+\frac{1}{A}e^{At}.$$ Then
$\varphi_w(z)=g(\chi(z))$. Thus
\begin{equation}\Delta \varphi_w =g''(\chi)|\nabla \chi|^2+g'(\chi)\Delta
\chi.\end{equation}  Since \begin{equation}\label{firstder}
g'(\chi)=e^{-A d(w(z))}
\end{equation}
and
\begin{equation}\label{secder} g''(\chi)=Ae^{-Ad(w(z))},
\end{equation}
it follows that \begin{equation}\label{subh}\Delta \varphi_w\ge
(A-(2\kappa_0+B)K^2)|\nabla \chi|^2e^{-Ad(u(z))}.\end{equation} In
order to have $\Delta \varphi_w\ge 0$, it is enough to take
\begin{equation}\label{A} A=(2\kappa_0+B)K^2 .\end{equation}
 Choosing $$\varrho = \max\{|z|:\mathrm{dist}(w(z),\gamma) =
\frac{1}{2\kappa_0}\},$$ then $\varphi_w$ satisfies the conditions
of the following generalization of E. Hopf lemma (\cite{hopf}):
\begin{lemma}\label{hopf}\cite{kalajan} Let $\varphi$ satisfies $\Delta \varphi\ge 0$ in
$R_\varrho=\{z:\varrho\le |z|<1\}$, $0<\varrho<1$, $\varphi$ be continuous
on $\overline{R_\varrho}$, $\varphi< 0$ in $R_\varrho$, $\varphi (t) = 0$ for
$t\in S^1$. Assume that the radial derivative $\frac{\partial
\varphi}{\partial r}$ exists almost everywhere at $t\in S^1$. Let
$M(\varphi,\varrho)=\max_{|z|=\varrho}\varphi(z)$. Then the inequality
\begin{equation}\label{ndihma}\frac{\partial \varphi(t)}{\partial
r}>\frac{2M(\varphi,\varrho)}{\varrho^2(1-e^{1/\varrho^2-1})}, \text{ for
a.e. } t\in S^1,\end{equation}
holds.
\end{lemma}
We will make use of \eqref{ndihma}, but under some improvement for
the class of q.c. harmonic mappings. The idea is to make the right
hand side of \eqref{ndihma} independent on the mapping $w$ for
$\varphi=\varphi_w$.

We will say that a q.c. mapping $f:\mathbf U\mapsto \Omega$ is
normalized if $f(1)=w_0$, $f(e^{2\pi/3 i}) = w_1$ and $f(e^{-2\pi/3
i}) = w_2$, where ${w_0w_1}$, $w_1w_2$ and $w_2w_0$ are arcs of
$\gamma=\partial \Omega$ having the same length $|\gamma|/3$.

In what follows we will prove that, for the class
$\mathcal{H}(\Omega,K,B)$ of  normalized $K$ q.c. mappings,
satisfying \eqref{pde} for some $B\ge 0$,  and mapping the unit disk
onto the domain $\Omega$, the inequality \eqref{ndihma}  holds
uniformly (see \eqref{ndihma2}).

Let $$\varrho:=\sup\{|z|:\mathrm{dist}(w(z),\gamma) =
\frac{1}{2\kappa_0}, w\in\mathcal{H}(\Omega,K,B)\}.$$

Therefore there exists a sequence $\{w_n\}$, $w_n\in
\mathcal{H}(\Omega,K, B)$ such that $$\varrho_n =
\max\{|z|:\mathrm{dist}(w_n(z),\gamma) = \frac{1}{2\kappa_0}\},$$
and $$\varrho = \lim_{n\to\infty}\varrho_n.$$

Notice now that, if $w_n$ is a sequence of normalized $K$-q.c.
mappings of the unit disk onto $\Omega$, then, up to some
subsequence, $w_n$ is a locally uniform convergent sequence
converging to some q.c. mapping $w\in \mathcal{H}(\Omega,K,B)$.
Under the condition on the boundary of $\Omega$, by
\cite[Theorem~4.4]{np} this sequence is uniformly convergent on
$\Bbb U$. Then there exists a sequence $z_n:
\mathrm{dist}(w_n(z_n),\gamma) = \frac{1}{2\kappa_0}$, such that,
$\lim_{n\to\infty}z_n = z_0$ and  $\varrho = |z_0|$. Since $w_n$
converges uniformly to $w$, it follows that, $\lim_{n\to
\infty}w_n(z_n) = w(z_0)$, and $\mathrm{dist}(w(z_0),\gamma) =
\frac{1}{2\kappa_0}.$ This infers $\varrho<1$.
\\
Let now
$$M(\varrho):=\sup\{M(\varphi_w,\varrho),  w\in\mathcal{H}(\Omega,K,B)\}.$$
Using the similar argument as above, we obtain that there exists a
uniformly convergent sequence $w_n$, converging to a mapping $w_0$,
such that
$$M(\varrho) =\lim_{n\to \infty} M(\varphi_{w_n},\varrho) = M(\varphi_{w_0},\varrho).$$
Thus $$M(\varrho)<0.$$ Setting $M(\varrho)$ instead of $M(\varrho,
\varphi)$ and $\varphi_{w}$ instead of $\varphi$ in \eqref{ndihma},
we obtain
\begin{equation}\label{ndihma2}\frac{\partial \varphi_w(t)}{\partial
r}>\frac{2M(\varrho)}{\varrho^2(1-e^{1/\varrho^2-1})}:=
C(K,\Omega,B), \text{ for a.e. } t\in S^1.\end{equation} To continue
observe that
$$\frac{\partial
\varphi_w(t)}{\partial r} =e^{Ad(w(z))}|\nabla d|
\left|\frac{\partial w}{\partial r}(t)\right|=e^{Ad(w(z))}
\left|\frac{\partial w}{\partial r}(t)\right| .$$ Combining
\eqref{A} and \eqref{ndihma2} we obtain $\text{
for a.e. } t\in S^1$
\begin{equation*}\left|\frac{\partial w}{\partial r}(t)\right| =
e^{-Ad(w(z))}\frac{\partial \varphi_w(t)}{\partial r}\ge
e^{-K^2}\frac{2M(\varrho)}{\varrho^2(1-e^{1/\varrho^2-1})}.\end{equation*}
The Lemma~\ref{lemica} is proved for normalized mapping $w$. If $w$
is not normalized, then we take the corresponding composition of $w$
and the corresponding M\"obius transformation, in order to obtain
the desired inequality. The proof of Lemma~\ref{lemica} is
completed.
\end{proof}
\begin{proof}[{\bf The finish of proof of Theorem~\ref{krye}}]
In this setting $w$ is harmonic and therefore $B=0$.

Assume first that ``$w\in
C^{1}(\overline {\mathbf U})$".

Let $l(\nabla w)(t)=||w_z(t)|-|w_{\bar z}(t)||$. As $w$ is $K$ q.c.,
according to \eqref{eq} we have
\begin{equation}\label{due}l(\nabla w)(t)\ge\frac{|\nabla w(t)|}{K}\ge
\frac{\abs{\frac{\partial w}{\partial r}(t)}}{K}\ge
\frac{C(K,\Omega,0,a_0)}{K}\end{equation} for $t\in S^1$. Therefore,
having in mind Lewy's theorem (\cite{hl}), which states that
$|w_z|>|w_{\bar z}|$ for $z\in \Bbb U$, we obtain for $t\in S^1$
that $|w_z(t)|\neq 0$ and hence:
$$
\frac{1}{|w_z|}\frac{C(K,\Omega,0,a_0)}{K} +\frac{|w_{\bar
z}|}{|w_z|}\le 1, \ t\in S^1.$$ As $w\in C^1(\overline {\mathbf
U})$, it follows that the functions
$$a(z):=\frac{\overline{w_{\bar z}}}{w_z},\quad
b(z):=\frac{1}{w_z}\frac{C(K,\Omega,0,a_0)}{K}$$ are well-defined
holomorphic functions in the unit disk having a continuous extension
to the boundary. As $|a|+|b|$ is bounded on the unit circle by 1, it
follows that it is bounded on the whole unit disk by 1 because
\[|a(z)|+|b(z)|\le P[|a|_{S^1}](z)+ P[|b|_{S^1}](z)=P[|a|_{S^1}+ |b|_{S^1}](z),\quad z\in\mathbb U. \]
This in turn implies that for every $z\in {\mathbf U}$
\begin{equation}\label{maj}l(\nabla w)(z) \ge  \frac{C(K,\Omega,0,a_0)}{K}=: C(\Omega,K, a_0) .\end{equation}
This infers that $$C(K,\Omega,a_0)\le
\frac{|w(z_1)-w(z_2)|}{|z_1-z_2|},\ \ \ \ z_1, z_2\in \mathbf U.$$
{Assume now that "$w
 \notin C^{1}(\overline {\mathbf U})$".} %This case can be done, using the similar approach as proof of \cite[Theorem~2.1.]{kalajan}.
We begin by this definition.
\begin{definition}
{\it Let $G$ be a domain in $\Bbb C$ and let $a\in\partial G$. We
will say that $G_a\subset G$ is a neighborhood of $a$ if there
exists a disk $D(a,r):=\{z:|z-a|<r\}$ such that $D(a,r)\cap
G\subset G_a$.}
\end{definition}
 Let $t=e^{i\beta}\in S^1$, then $w(t)\in
 \partial \Omega$. Let $\gamma$ be an arc-length parametrization of $\partial\Omega$
with $\gamma(s)=w(t)$. Since $\partial \Omega \in C^{1,1}$, there
exists a neighborhood $\Omega_t$ of $w(t)$ with $C^{1,1}$ Jordan
boundary such that,
\begin{equation}\label{subdomains}\Omega^\tau _t:=\Omega_t +
i\gamma'(s)\cdot \tau\subset \Omega,\text{ and } \partial
\Omega^\tau_t\subset \Omega \text{ for $0< \tau\le \tau_t$\,
($\tau_t>0$) }.\end{equation} An example of a family $\Omega^\tau_t$
such that $\partial\Omega^\tau_t \in C^{1,1}$ and with the property
\eqref{subdomains} has been given in \cite{kamz}.

Let $a_t\in \Omega_t$ be arbitrary. Then $a_t+i\gamma'(s)\cdot
\tau\in \Omega_t^\tau$. Take $U_\tau = f^{-1}(\Omega_t^\tau)$. Let
$\eta_t^\tau $ be a conformal mapping of the unit disk onto $U_\tau$
such that $\eta_t^\tau(0)=f^{-1}(a_t +i\gamma'(s)\cdot \tau)$, and
$\arg\frac{ d  \eta_t^\tau}{ d  z}(0)=0$. Then the mapping
$$f_t^\tau(z) := f(\eta_t^\tau(z))-i\gamma'(s)\cdot \tau$$ is a
harmonic $K$ quasiconformal mapping of the unit disk onto $\Omega_t$
satisfying the condition $f_t^\tau(0)= a_t$. Moreover
$$f_t^\tau\in C^1(\overline {\mathbf U}).$$ Using the {case
"$w\in C^1(\overline {\mathbf U})$"}, it follows that $$|\nabla
f_t^\tau (z)|\ge C(K,\Omega_t, a_t).$$ On the other hand
$$\lim_{\tau\to 0+}\nabla f_t^\tau(z) = \nabla (f\circ \eta_t)(z)$$
on the compact sets of ${\mathbf U}$ as well as $$\lim_{\tau \to
0+}\frac{ d  \eta_t^\tau}{ d  z}(z)= \frac{ d \eta_t}{ d z}(z),$$
where $\eta_t$ is a conformal mapping of the unit disk onto
$U_0=f^{-1}(\Omega_t)$ with $\eta_t(0)=f^{-1}(a_t)$. It follows that
$$|\nabla f_t(z)|\ge C(K,\Omega_t, a_t).$$

By using the Schwarz's reflexion principle to the mapping $\eta_t$,
and using the formula $$\nabla (f\circ \eta_t)(z) = \nabla f\cdot
\frac{ d  \eta_t}{ d  z}(z)$$ it follows that in some neighborhood
$\tilde U_t$ of $t\in S^1$ with smooth boundary ($D(t,r_t)\cap
\mathbf U\subset \tilde U_t$ for some $r_t>0$), the function $f$
satisfies the inequality
\begin{equation}\label{local1} |\nabla f(z)|\ge \frac{C(K,\Omega_t,
a_t)}{\max\{|\eta'_t(\zeta)| : \zeta\in  \overline{\tilde
U_t}\}}=:\tilde C(K,\Omega_t,a_t)>0.
\end{equation}
Since $S^1$ is a compact set, it can be covered by a finite family
$\partial\tilde U_{t_j}\cap S^1\cap D(t,r_t/2)$, $j=1,\dots,m$. It
follows that the inequality
\begin{equation}\label{local} |\nabla f(z)|\ge
\min\{\tilde C(K,\Omega_{t_j}, a_{t_j}): j=1,\dots, m\}=:\tilde
C(K,\Omega,a_0)>0,
\end{equation} there holds in the annulus $$\tilde R=\left\{z:1-\frac{\sqrt 3}{2}\min_{1\le j\le m}r_{t_j}<|z|<1\right\}
\subset \bigcup_{j=1}^m \tilde U_{t_j}.$$
 This implies that the subharmonic function $S=|a(z)|+|b(z)|$ is bounded in $\mathbf U$. According to the maximum
principle, it is bounded by $1$ in the whole unit disk. This in turn
implies again \eqref{maj} and consequently
$$\frac{C(K,\Omega,a_0)}{K}|z_1-z_2|\le |w(z_1)-w(z_2)|,\ \ \ \
z_1, z_2\in \mathbf U.$$
\end{proof}
\subsection*{Acknowledgment} I thank the referee for providing constructive comments and
help in improving the contents of this paper.


\begin{thebibliography}{1}


\bibitem{Ahl} Ahlfors, L. \textit{Lectures on Quasiconformal mappings,}
Van Nostrand Mathematical Studies, D. Van Nostrand 1966.

\bibitem{ale}
Alessandrini, G.; Nesi, V.  {\it Invertible harmonic mappings,
beyond Kneser,} Ann. Scuola Norm. Sup. Pisa, Cl. Sci. (5) VIII
(2009), 451-468.

\bibitem{akm}
Arsenovic, M.; Kojic, V.; Mateljevic, M. {\it On lipschitz
continuity of harmonic quasiregular maps on the unit ball in $\Bbb
R^n$}, Ann. Acad. Sci. Fenn., Math. Vol {\bf 33}, 315-318, (2008).



\bibitem{lit}
Li, P.; Tam, L. {\it Uniqueness and regularity of proper harmonic
maps.} Ann. of Math. (2) {\bf 137} (1993), no. 1, 167--201.

\bibitem{lit1}
\bysame  {\it Uniqueness and regularity of proper harmonic maps.
II.} Indiana Univ. Math. J.  {\bf 42}  (1993),  no. 2, 591--635.

\bibitem{G}
Goluzin, G. M. {\it Geometric function theory}, Nauka Moskva 1966
(Russian).

\bibitem{gt} Gilbarg, D.; Trudinger. N. \emph{Elliptic Partial
Differential Equations of Second Order}, Vol. {\bf 224}, 2 Edition,
Springer 1977, 1983.

\bibitem{hs}
Hengartner, W.; Schober, G. {\it Harmonic mappings with given
dilatation.}  J. London Math. Soc. (2)  {\bf 33}  (1986),  no. 3,
473--483.

%\bibitem{HE}
%Heinz, E. {\it On one-to-one harmonic mappings. } Pac. J. Math. {\bf
%9}, 101-105 (1959).
%
%\bibitem{EH} Heinz, E. {\it On certain nonlinear elliptic
%differential equations and univalent mappings,} J. d' Anal. {\bf 5},
%1956/57, 197-272.
%
\bibitem{hopf} Hopf, E. {\it A remark on linear elliptic
differential equations of second order}, Proc. Amer. Math. Soc., 3,
791-793  (1952).

\bibitem{kalajan} Kalaj, D. \emph{Lipschitz spaces and harmonic
mappings}, Ann. Acad. Sci. Fenn., Math. 2009 Vol {\bf 34}.
(arXiv:0901.3925v1).

\bibitem{kalajpub}
 \bysame: \emph{Quasiconformal harmonic functions between convex
domains}, Publ. Inst. Math., Nouv. Ser. {\bf 76}(90), 3-20 (2004).

\bibitem{mz} \bysame: {\it On harmonic quasiconformal self-mappings of the unit
ball}, Ann. Acad. Sci. Fenn., Math. Vol {\bf 33}, 1-11, (2008).

\bibitem{kamz} \bysame: {\it Quasiconformal harmonic mapping
between Jordan domains}  Math. Z. Volume {\bf 260}, Number 2,
237-252, 2008.

\bibitem{Dk}
\bysame: {\it On harmonic diffeomorphisms of the unit disc onto a
convex domain.} Complex Variables, Theory Appl. {\bf 48}, No.2,
175-187 (2003).

\bibitem{Kalaj} \bysame: {\it On quasiregular mappings between smooth Jordan domains.}  J. Math. Anal. Appl.  {\bf 362}  (2010),  no. 1, 58--63..

\bibitem{km} Kalaj, D.; Mateljevi\'c, M.
{\it Inner estimate and quasiconformal harmonic maps between smooth
domains}, Journal d'Analise Math. {\bf 100.} 117-132, (2006).
\bibitem{kalmat}

\bysame: {\it On certain nonlinear elliptic PDE and quasiconfomal
mapps between Euclidean surfaces.} To appear in Potential Analysis.
DOI 10.1007/s11118-010-9177-x.

\bysame: {\it On quasiconformal harmonic surfaces with rectifiable
boundary.} To appear in Complex Analysis and Operator Theory. DOI:
10.1007/s11785-010-0062-9.


\bibitem{KP} Kalaj, D; Pavlovi\'c, M.
{\it Boundary correspondence under harmonic quasiconformal
homeomorfisms of a half-plane,} Ann. Acad. Sci. Fenn., Math. {\bf
30}, No.1, (2005) 159-165.

\bibitem{trans} \bysame:
\textit{On quasiconformal self-mappings of the unit disk satisfying
the Poisson's equation}, to appear in Transaction of AMS.
%\bibitem{knes} H. Kneser: {\it L\"osung der Aufgabe 41}, Jber. Deutsch.
%Math.-Verein. 35 (1926), 123-124.

\bibitem{MMM}
Knezevic, M.; Mateljevic. M. {\it On the quasi-isometries of
harmonic quasiconformal mappings} Journal of Mathematical Analysis
and Applications, 2007; {\bf 334} (1) 404-413.

\bibitem{lv}
Lehto O.; Virtanen, K.I.{\it Quasiconformal mapping},
Springer-verlag, Berlin and New York, 1973.

\bibitem{hl}
Lewy, H. {\it On the non-vanishing of the Jacobian in certain in
one-to-one mappings,} Bull. Amer. Math. Soc. {\bf 42}. (1936),
689-692.

\bibitem{kojic}
Manojlovi\'c, V: {\it Bi-lipshicity of quasiconformal harmonic
mappings in the plane.} Filomat {\bf 23}:1 (2009), 85–-89.

\bibitem{Om}
Martio, O, {\it On harmonic quasiconformal mappings}, Ann. Acad.
Sci. Fenn., Ser. A I {\bf 425} (1968), 3-10.


\bibitem{mv}
Mateljevic, M.; Vuorinen M. {\it On harmonic quasiconformal
quasi-isometries}, to appear in Journal of Inequalities and
Applications (http://www.hindawi.com/journals/jia/, articles in
Press).

\bibitem{np}
N\"akki, R.; Palka, B. {\it Boundary regularity and the uniform
convergence of quasiconformal mappings.} Comment. Math. Helv. {\bf
54} (1979), no. 3, 458--476.

\bibitem{pk} Partyka D.; Sakan, K.
{\it On bi-Lipschitz type inequalities for quasiconformal harmonic
mappings,} Ann. Acad. Sci. Fenn. Math.. Vol {\bf 32}, pp. 579-594
(2007).

\bibitem{MP} Pavlovi\' c, M.
{\it Boundary correspondence under harmonic quasiconformal
homeomorfisms of the unit disc}, Ann. Acad. Sci. Fenn., Vol {\bf
27}, (2002) 365-372.
%\bibitem{rado}
%T. Rad\'o: Aufgabe 41, Jber. Deutsch. Math.-Verein. 35 (1926), 49.
\bibitem{w} Pommerenke, C. {\it Boundary behavour of conformal maps,}
Springer-Verlag, New York, 1991.
\bibitem{wang}
Wang, C. {\it A sharp form of Mori's theorem on Q-mappings,} Kexue
Jilu, {\bf 4} (1960), 334-337.

\bibitem{wan} T. Wan, {Constant mean curvature surface, harmonic
maps, and universal Teichm\"uller space}, J. Diff. Geom. {\bf 35}
(1992) 643-657.

\bibitem{w1} Warschawski, S. E. {\it On differentiability at the boundary in
conformal mapping,} Proc. Amer. Math. Soc, {\bf 12} (1961), 614-620.

\bibitem{w2} \bysame {\it On the higher derivatives at the boundary in conformal
mapping,} Trans. Amer. Math. Soc, {\bf 38}, No. 2 (1935), 310-340.

\end{thebibliography}
\end{document}